\documentclass[final,1p,times,english,sort&compress]{elsarticle}
\usepackage{color,latexsym,amsmath,amsfonts,amssymb,xcolor}
\usepackage[utf8]{inputenc}
\usepackage{layout,verbatim,amssymb,amsmath,setspace,color,cleveref,enumerate,fontenc}
\usepackage{amsthm}
\biboptions{sort&compress}
\usepackage[mathscr]{eucal}
 \numberwithin{equation}{section}

\newtheorem{theorem}{~Theorem}[section]
\newtheorem{lemma}[theorem]{~Lemma}
\newtheorem{proposizione}[theorem]{~Proposition}
\newtheorem{corollario}[theorem]{~Corollary}
\newtheorem{osservazione}[theorem]{~Remark}
\newtheorem{definition}[theorem]{~Definition}

\newtheorem{ass}[theorem]{~Assumption}
\newtheorem{rem}[theorem]{~Remark}

\newcommand{\iS}{\mbox{$\displaystyle{\int^*_{S}}$}}

\newcommand{\vuoto}{{\rm \O}}

\newcommand{\erre}{\mathbb{R}}

\newcommand{\Bs}{\mbox{$\mathcal B$~}}
\newcommand{\As}{\mbox{$\mathcal A$}}
\newcommand{\Hs}{\mbox{$\mathcal H$~}}

\newcommand{\Es}{\mbox{$\mathcal E$}}
\newcommand{\WEs}{\mbox{$W(\Es)$}}
\newcommand{\CEs}{\mbox{$C(\Es)$}}
\newcommand{\LCEs}{\mbox{$LC(\Es)$}}

\begin{document}
\begin{frontmatter}
\title{\bf A special class of fuzzy measures: Choquet integral and applications}
\author[add1]{Domenico Candeloro}\ead{domenico.candeloro@unipg.it}
\author[add2]{Radko Mesiar\corref{cor2}} \ead{mesiar@math.sk}
\cortext[cor2]{corresponding author address: Slovak University of Technology (STU) Bratislava Faculty of Civil Engineering, Department of Mathematics Radlinského 11, 810 05 Bratislava, (Slovakia)}
\author[add1]{Anna Rita Sambucini} \ead{anna.sambucini@unipg.it} 
\address[add1]{Department of Mathematics and Computer Sciences, University of Perugia, Via Vanvitelli 1, 06123 Perugia (Italy)}
\address[add2]{Slovak University of Technology (STU) Bratislava Faculty of Civil Engineering, Department of Mathematics Radlinského 11, 810 05 Bratislava, (Slovakia)}
\date{today}
\begin{abstract}
Core of an economy and Walras equilibria are considered  for a product space $X \times [0,1]$ using the Choquet integral with respect to a fuzzy measure.
\end{abstract}
\end{frontmatter}
\section{Introduction}\label{introduzione} 
The Choquet integral, introduced by G. Choquet in 1953 (\cite{Choquet}), gives a method to integrate functions with respect to  non necessarily additive measures such as capacities or, more generally, fuzzy measures (\cite{BM2000,BM2000a,BMV2002,bs9596,bs97,JKK,JK,LIMP,LIMP2,mesiar,MuSu,pap1}).
Besides its initial applications in potential theory and statistical mechanics it became a useful tool to deal with uncertainty in imprecise probability theory, in decision theory  and in the study of cooperative games (\cite{cv2013,cpv2014,z,m-altri,nt}).
The Choquet integral has applications also in finance, economics and insurance.
One of the central problems in Mathematical Economics is the search of equilibria for the model;
in the finitely additive framework results are given in \cite{AM1,AM4,AR,B,BG,BDG1,BDG13,V}.
In this research, we assume that the space of agents is decomposed into a large number of {\em sections}, each of 
which is an authonomous economic model, but coalitions can be created also among members of different sections, according with some rules.
\\
The mathematical model is a {\em product space} $X^*:=X\times [0,1]$, where the {\em sections} are the sets 
$X\times \{y\}$, as $y$ ranges in $[0,1]$. The set $X$ represents the {\em typical} section of agents,
and is endowed with a $\sigma$-algebra $\As$, 
while in the $[0,1]$ space the usual Lebesgue $\sigma$-algebra $\Bs$ and measure 
$\lambda$ are fixed. 
In each section $X\times \{y\}$ the obvious $\sigma$-algebra $\As\times\{y\}$ is considered, 
with a fuzzy measure $\mu_y$ defined there.  
\\
The {\em coalitions} are all the sets of the product $\sigma$-algebra $\Hs$ generated by $\As$ and the Borel $\sigma$-algebra $\Bs$,
and the fuzzy measure $m:\Hs\to \erre^+_0$ is defined by integrating the measures $\mu_y$ with respect to the 
Lebesgue measure in $[0,1]$.
This model includes 
 the case in which the {\em sections} are just a finite number of sets, $E_1,...,E_r$, 
and $\mu$ is additive along them. Indeed, in that case it will suffice to define $X$ as the union of these sets, and to decompose $[0,1]$ into $r$ subintervals $J_i$ (such that $\lambda(J_i)=\mu(E_i)$),  in each point $y$ of which the measure  $\mu_y$ is null outside the
set $E_i\times \{y\}$ and coincides with $\dfrac{\mu}{\mu(E_i)}$ in the measurable subsets of $E_i\times\{y\}$.
\section{Preliminaries and definitions}\label{preliminari}
In $\erre ^n$ we shall denote by $\mathbb{R}_+^n$ the positive orthant, and by $(\mathbb{R}_+^n)^{\circ}$
 its interior.
Also we shall denote by $\leq $ the usual order between numbers, and by \underline{$\ll $} the usual
partial order between vectors in $\erre ^n$.
Let $(X, \mathcal{A})$ be a measurable space.
\begin{definition}\rm (Murofushi and Sugeno \cite{MuSu})
A fuzzy measure on a measurable space $(X, \mathcal{A})$  is a set function 
$\mu: \mathcal{A} \rightarrow \mathbb{R}_0^+$ with the properties:
\begin{itemize}
\item $\mu(\emptyset) = 0$; \hskip.3cm  $\mu (X) < +\infty$;
\item if $A \subset B$, then 
$\mu(A) \leq \mu(B)$  (monotonicity).
\end{itemize}
A fuzzy measure $\mu$ is {\em subadditive} if 
$\mu(A\cup B)\leq \mu(A)+\mu(B)$
for all elements $A,B$ from $\As$. A subadditive fuzzy measure will be also called a {\em submeasure}.
\end{definition}
Submeasures are also called capacities, for an overview of the topic see for example \cite{pap1, bb2017}.\\

We now recall the concept of a {\em semiconvex} submeasure: 
\begin{definition}\label{CM2.1} \rm (See also  
 \cite[Lemma 2.1]{CM1978})
\rm If $\mu :\mathcal{A} \rightarrow [0,1]$ is a fuzzy
submeasure, we say that it is {\em semiconvex} if
 for every $A\in \mathcal{A}$ there exists a family of subsets of $A$,
$(A_t)_{t\in [0,1]}\subset \mathcal{A} $ such that
\begin{description}
\item[(\ref{CM2.1}.i)] $A_0 = \vuoto,~ \, A_1 = A$;
\item[(\ref{CM2.1}.ii)] $\mu (A_t) = t\mu (A)$;
\item[(\ref{CM2.1}.iii)] for $t < t'$, there hold $A_t \subset A_{t'}$ and $\mu(A_{t'} \setminus A_t) = (t' -t) \mu(A)$.
\end{description}
\end{definition}

\begin{definition}\label{filtrazioneastratta}
\rm
Given a fuzzy measure $\mu: \mathcal{A} \rightarrow \mathbb{R}_0^+$, (not necessarily sub-additive), we say that it is {\em filtering} if, for every element $A\in \mathcal{A}$ there exists an increasing family
$(A_t)_{t\in [0,1]}$ of measurable subsets of $A$, such that
{\bf (\ref{CM2.1}.i)-(\ref{CM2.1}.iiii)} above hold true.
\end{definition}

\begin{rem}\rm \label{rem-filtr}
In this case, the range of $\mu_A$, namely of the measure $\mu$ restricted to $A$   is $[0,\mu(A)]$,
for every measurable $A\subset X$.
Moreover, given a family $\{\mu_y, y\in Y\}$ of fuzzy measures on $\mathcal{A}$, we say that they are {\em uniformly filtering} if for each element $A\in \mathcal{A}$ the same filtering family can be found, for all $\mu_y$.

A nontrivial example can be given as follows: Let $X_0=[0,1]$, and assume that $\mathcal{A}$ is the family of all subsets of $X_0$.
 It is well-known that there exist additive positive measures on $X_0$ (not $\sigma$-additive) extending the Lebesgue measure $\lambda$ to $\mathcal{A}$:  see e.g. \cite[Theorem 1.3]{CM1978} and related bibliography \cite{es,na}; similar questions were discussed also in  \cite{colreg}. Denote by $\mu_0$
 any of these measures, and set, for every $y\in ]0,1]$, $\mu_y(A)=\mu_0(A)^y$: clearly, $\mu_y$ is a fuzzy measure on $X_0$,  not additive in general and it is filtering, because of continuity. Now, let $X=X_0\times ]0,1]$ and define, for each $E\subset X$
$$\mu^*_y(E)=\mu_y(E_y),$$
where $E_y:=\{x\in X_0:(x,y)\in E)\}$. Clearly, for each $t\in [0,1 ]$, each $y$ and every $E$ there exists a subset 
$E'_{t,y}\subset E_y$ such that $\mu_y(E'_{t,y})=t\mu_y(E_y)=t\mu_y^*(E)$. Thus, setting 
$E^*_t=\bigcup_{y\in ]0,1]}E'_{t_y}$
we get
$$\mu^*_y(E^*_t)=\mu_y(E'_{t,y})=t\mu_y^*(E),$$
as requested for uniform filtering.\\
We might be wondering if any classical $\sigma$-additive measure $\mu$ with range $[0,\mu(X)]$ has a filtering family satisfying  (\ref{CM2.1}.iiii). The answer is negative, since 
this condition for scalar measure is strictly related to the notion of continuity 
(\cite[Definition 1.2]{CM1979}:  for every $\varepsilon > 0$ there exists a finite partition of $X : E_1, E_2, \ldots E_k$ 
such that $\mu(E_i) \leq \varepsilon$ for every $i \leq k$)
and it does not follow  in general from the additivity of the measure. In fact, if we consider $X = \mathbb{N}$ and the measure $\mu: \mathcal{P}(\mathbb{N}) \to [0,1]$ defined by $\mu(E) = \sum_{n =1}^{\infty} \dfrac{1_E (n)}{2^n}$;  this measure
 is $\sigma$-additive with arcwise-connected range ($R(\mu) = [0,1]$) but it is not continuous in the sense of \cite[Definition 1.2]{CM1979} and so the condition (\ref{CM2.1}.iiii) is not fulfilled. If the target space is infinite dimensional the situation 
is even worse, in fact the continuity does not imply the semiconvexity and then the convexity of the range (\cite{CM1979}).
\end{rem}


\section{The Choquet integral  and its properties}\label{choquet} 

Let $(X, \mathcal{A}, \mu)$ be a 
fuzzy measure space. 
\begin{definition}\label{def-ch} \rm A  function $f:X \rightarrow \erre^+_0$  is said to be {\em  measurable}
 if the set $\{x\in X | f(x) > t\}$ is in $\mathcal{A }$ for every $t > 0$. Any set of that type will be often denoted as $[f>t]$.
The  \it Choquet integral \rm of a measurable function $f$ is defined by 
\[ \int f d\mu:= \int_0^{\infty} \mu( [f > t]) dt;\]
where the latter integral is in the Riemann sense.
We say that  $f \in L^1_C(\mu)$ if and only if $f$ is measurable and  
$\int f d\mu < \infty$.
\end{definition}
\noindent
This integral fulfills the following properties (\cite[Proposition 5.1  and Chapter 11]{D})
\begin{description}
\item[(\ref{def-ch}.i)] $\int 1_A~ d\mu= \mu(A)$;
\item[(\ref{def-ch}.ii)] $\int c f d\mu= c \int f d\mu$ for $c \geq 0$;
\item[(\ref{def-ch}.iii)] if $f \leq g$ then $\int f d\mu \leq \int g d\mu$;
\item[(\ref{def-ch}.iv)] $\int (f + c) d\mu= \int f d\mu+ c\mu(X)$ for every $c \in \erre^+$.
\item[(\ref{def-ch}.v)] if $\mu$ is subadditive then
\[\int (f+g) d\mu\leq \int f d\mu+ \int g d\mu; \hskip.5cm
\mbox{(\cite[Theorem 6.3]{D}) }
\]
\item[(\ref{def-ch}.vi)] 
if $f,g$ are {\em comonotonic}, namely there is no pair $x, y \in X$ such that
$f(x) > f(y)$ and $g(x) < g(y)$,  then 
\begin{align*}
\int (f+g) d\mu = \int f d\mu+ \int g d\mu; \hskip.5cm
\mbox{(comonotonic additivity) }
\end{align*}
\item[(\ref{def-ch}.vii)] for every $c > 0$ it is
\begin{align*}
\int f d\mu &= \int \min\{f,c\} d\mu+ \int (f- \min\{f,c\}) d\mu; \\
&\mbox{(horizontal additivity) }
\end{align*}
\item[(\ref{def-ch}.viii)] for every non negative and $\mathcal{A}$-measurable function $f$ it is
\[ \int_A f d\mu := 
\int f d\mu_A 
= \int f 1_A d\mu\]
(see \cite[chapter 11]{D}); for arbitrary $f$ this last equality fails, as showed in \cite[Example 11.1]{D}.
\end{description}

There is a huge literature concerning (\ref{def-ch}.vi) and its consequences; an interesting result on additivity is contained 
in \cite{mes-sip} where the \v{S}ipo\v{s} and the Choquet integrals are compared and 
the additivity of the integrals are examined on some subspaces.\\

Now, we shall introduce a class of fuzzy measures, that in some sense can be considered as {\em averages} 
of other fuzzy measures on the space $X$. In particular, we shall assume the following, which will be kept for all the sequel.
\begin{definition}\label{evoluzione}\rm
Let $X^*:=X\times [0,1]$, and in $X^*$ let $\mathcal H$ be the $\sigma$-algebra $\mathcal{A}\times \mathcal{B}$,
 i.e. the product $\sigma$-algebra obtained by $\mathcal{A}$ and the Borel $\sigma$-algebra $\mathcal{B}$ in $[0,1]$. 
We say that a fuzzy measure $m:\mathcal H\to \erre^+_0$ is {\em decomposable} if, 
for every real number $y\in [0,1]$ there exists a non-trivial fuzzy measure $\mu_y:\mathcal{A}\to \erre^+_0$, in such 
a way that the measures $\mu_y$ turn out to be equibounded 
and that
$$m(H)=\int_0^1 \mu_y(H_y)dy$$
holds true, for all $H\in \mathcal{H}$.
(Here the set $H_y$ is the $y$-section of $H$, and we implicitly assume that the mapping $y\mapsto \mu_y(H_y)$ is a
 measurable map, for all $H$).
\end{definition}

From now on the measure $m$ will be decomposable.
\begin{definition}\label{sezional}\rm
Let $X^*:=X\times [0,1]$ be as above, and $f:X^*\to \erre^+_0$ be any mapping. We say that $f$ is {\em sectional}
if there exists a measurable mapping $\varphi:[0,1]\to \erre^+_0$ such that 
$$f(x,y)=\varphi(y),$$
for all $x$ and all $y$.
If this is the case, we say that $\varphi$ is the {\em section function} of $f$. Usually, when this is the case, we shall also write $f(y)$ rather than $f(x,y)$, thus identifying $f$ and $\varphi$.
\end{definition}

A kind of Fubini Theorem can be deduced for an arbitrary non-negative integrable mapping $f$, asserting that the integral
 of $f$ is obtained as an iterated one.
We first prove a technical result of joint measurability for real-valued functions.
\begin{lemma}\label{joint}
Let us assume that $g:[a,b]\times [c,d]\to \erre$ is any mapping, satisfying the following conditions:
\begin{description}
\item[(\ref{joint}.i)] $g(t, \cdot)$ is measurable for all $t\in [a,b]$;
\item[(\ref{joint}.ii)]  there exists a finite measurable partition $\{E_j,j=1...k\}$ of the interval $[a,b]$ such that
 $g(t,y)=g(t',y)$ for all $t,t'\in E_j$ and all $y$.
\end{description}
Then $g$ is jointly measurable in $(t,y)$.
\end{lemma}
\begin{proof}
For each index $j$ from $1$ to $k$ choose arbitrarily a point $t_j\in E_j$.
For each positive $\tau$, set
$$A(\tau):=\{(t,y):g(t,y)>\tau\}:$$
for each fixed $y$, the $y$-section $A(\tau)_y$ is an element of the finite algebra in $[a,b]$ generated by the sets $E_j$. 
\\
Now, denoting by $F_1,...F_K$ the elements of this algebra, and 
 setting $Y_h:=\{y: A(\tau)_y=F_h\}$, $h=1,...,K$, we can see that
$Y_h$ is measurable since it is a finite intersections of sets of the type $\{y\in [c,d]:g(t_i,y)>\tau\}$ for suitable
values of $i$ and of the opposite type $\{y\in [c,d]:g(t_i,y)\leq \tau\}$ for the other indexes $i$, and in turn
 these sets are measurable since the mapping $g$ is separately measurable for each fixed $t$.
\\
 Finally, the formula
$$A(\tau)=\bigcup_{h=1}^k F_h\times Y_h$$
shows measurability of the set $A(\tau)$.
\end{proof}

\begin{theorem}\label{fubini2}
Let $f:X^*\to \erre^+_0$ be any integrable map. Then we have
$$\int_{X^*}fdm=\int_0^1\left(\int_X f(x,y)d\mu_y(x) \right)dy.$$
\end{theorem} 
\begin{proof}
 We first assume that $f$ is bounded: $f(x,y)\leq M$ for all $(x,y)\in X^*$. For  $(t,y)\in [0,M]\times[0,1],$ 
consider the mapping 
$$g(t,y):=\mu_y(\{x: f(x,y)>t\}).$$
 Of course, $g$ is decreasing in $t$.
By means of dyadic partitions of $[0,M]$, (say $\{[t_i^n,t_{i+1}^n[, i=0,...,2^n-1\}$, where 
$t_i^n=\dfrac{i}{2^n}M$), it is easy to construct two sequences, $(g_n)_n$ and $(g^n)_n$ of {\em step}
 functions, such that for each $n$ and $y$ the mapping $g_n(t;y)$ is constant in each dyadic interval, and 
equal to the value of $g(\cdot,y)$ at the right endpoint of the interval, while $g^n(t;y)$ equals the value of 
$g(\cdot,y)$ at the left endpoint. 
In such a way, we have 
$$g_n(t;y)\leq g_{n+1}(t;y)\leq g(t,y)\leq g^{n+1}(t;y)\leq g^n(t;y)$$
for all $n$, $y$ and $t$.
We also point out that the mappings $g_n$ and $g^n$ (considered as depending on $t$ and $y$), are 
$\mathcal{B}^2$-measurable: this  follows from Lemma \ref{joint}, since (\ref{joint}-ii) is satisfied by 
construction an (\ref{joint}-i)
 follows from integrability of $f$ and definition of $m$.
\\
Since $g$ is continuous in $t$ for all $t$ except a countable set, we can deduce that both sequences $(g_n)_n$ 
and $(g^n)_n$ converge to $g$ except for a countable set of values $t$ (possibly depending on $y$).
Then, denoting by $\underline{g}$ and $\overline{g}$ respectively the limits of $g_n$ and $g^n$, and using 
dominated convergence, we see that
$$\int_0^M\underline{g}(t;y)dt=\int_0^Mg(t;y)dt=\int_0^M\overline{g}(t;y)dt$$
holds, for all $y$.
Moreover, the functions $\underline{g}$ and $\overline{g}$ are $\mathcal{B}^2$-measurable, as limits of 
sequences of mappings of this type.
Thanks to Fubini's Theorem and to convergence in $L^1$, we now deduce that
\begin{eqnarray*}
\int_{X^*}f(x,y) dm &=&  \int_0^M\left(\int_0^1g(t;y)dy\right) dt\geq
\int_0^M\left(\int_0^1 \underline{g}(t;y)dy\right)dt =
\\ &=& 
\int_0^1\left(\int_0^M \underline{g}(t;y)dt\right)dy=
\int_0^1\left(\int_0^M g(t;y)dt\right)dy=
\\ &=&
\int_0^1\left(\int_Xf(x,y)d\mu_y(x)\right) dy.
\end{eqnarray*}
On the other hand,
\begin{eqnarray*}
 \int_{X^*}f(x,y) dm &=& \int_0^M\left(\int_0^1g(t;y)dy\right) dt
\leq\int_0^M\left(\int_0^1\overline{g}(t;y)dy\right)dt=\\
&=&
 \int_0^1\left(\int_0^M \overline{g}(t;y)dt\right)dy=
\int_0^1\left(\int_0^M g(t;y)dt\right)dy=
\\ &=& \int_0^1\left(\int_Xf(x,y)d\mu_y(x)\right) dy.
\end{eqnarray*}
Comparing the two inequalities found, we obtain the assertion, for bounded $f$.\\
Now, if $f$ is unbounded, it is easy to reach the conclusion, by setting
$f_n=f\wedge n$ for each integer $n$, and observing that
$$\int_{X^*}f dm=\lim_n\int_{X^*}f_n dm$$
and also
$$\int_{X}f(x,y) d\mu_y(x)=\lim_n\int_{X}f_n(x,y) d\mu_y(x)$$
for each fixed $y$, from which
$$\int_0^1\left(\int_Xf(x,y)d\mu_y(x)\right) dy=\lim_n\int_0^1\left(\int_X f_n(x,y)d\mu_y(x)\right) dy,$$
by monotone convergence.
\end{proof}

\begin{corollario}\label{fubini1}
Assume that $f:X^* \to \mathbb{R}^+_0$ is sectional and integrable.
Then we have, for each $H\in \mathcal{H}$:
$$\int_{H}f dm=\int_0^1f(y)\mu_y(H_y)dy,$$ 
where as usual $H_y$ denotes the $y$-section of $H$.
\end{corollario}  
\begin{proof}
Let $f:X^*\to \erre^+_0$ be any sectional integrable map, $f(x,y)=f(y)$, and choose arbitrarily any measurable set 
$H\subset X^*$. Then 
$$\int_{H} f dm=\int_{X^*}f(y)1_H(x,y)dm.$$
Thanks to Theorem \ref{fubini2}, we get
$$\int_{H} f dm=\int_0^1\left(\int_{H_y}f(y)d\mu_y(x)   \right)dy,$$
where $H_y=\{x\in X:(x,y)\in H\}$. Since the inner integrand is independent on $x$, we obtain
$$\int_{H} f dm=\int_0^1 \mu_y(H_y)f(y)dy,$$
as announced.
\end{proof}

\begin{osservazione}\label{add} 
\rm
Another consequence of Theorem \ref{fubini2} is the following.

If two integrable  functions $f_1,f_2$ are comonotonic with respect to $x$ for every $y \in [0,1]$, then
\begin{eqnarray}
\label{com-add}
\int_{X^*} (f_1 + f_2) dm  &=& \int_0^1 \left(\int_X [f_1(x,y) + f_2(x,y) ]d\mu_y(x)\right) = \\
&=& \nonumber  \int_{X^*} f_1 dm + \int_{X^*} f_2 dm.
\end{eqnarray}
In particular, If $f_1, f_2$ are sectional and integrable functions and $f_1 \pm f_2 \geq 0$
 then, for every $H \in \mathcal{H}$
\begin{eqnarray}
\label{sec-add}
 \int_H ( f_1 \pm  f_2 ) dm =  \int_H f_1 dm \pm  \int_H f_2 dm.
\end{eqnarray}
\end{osservazione}

For measurable  vector functions $f : X^* \rightarrow \mathbb{R}_+^n$, the Choquet integral is defined componentwise,
 so it is an $n$-dimensional vector too. Assuming that $m$ is decomposable,
we first prove the following result:
\begin{proposizione}\label{perscalari}
For every constant vector $p\in \erre^n_+$ and every  integrable sectional function $f:X^*\to \erre^n_+$, we have
$$\int_{X^*} p \cdot f dm = p \cdot \int_{X^*}fdm.$$
\end{proposizione}
\begin{proof}
If $f$ is sectional then its components are sectional too and so it is enough to apply  (\ref{sec-add}) of Remark \ref{add}.
\end{proof}

The additivity obtained in Propostion \ref{perscalari} can be extended 
to functions $f(x,y) = g(x) h(y)$ for suitable $g$ and $h$.
\begin{proposizione}\label{perscalari2}
For every constant vector $p\in \erre^n_+$, every bounded measurable function $g:X\to \erre^+_0$,  and every  integrable vector function $h:[0,1]\to \erre^n_+$, we have
$$\int_{X^*} g(x) p \cdot h(y) dm = p \cdot \int_{X^*}g(x)h(y)dm.$$
\end{proposizione}
\begin{proof}
First, we observe that
 the conclusion can be easily obtained, when $g=c1_H$, where $c$ is any positive real constant and $H$ is any 
measurable subset of $X$, i.e.
\begin{eqnarray}\label{primostep}
\int_{X\times [0,1]} c1_H(x) p \cdot h(y) dm = p \cdot \int_{X\times [0,1]} c 1_H(x) h(y) dm.
\end{eqnarray} 
Now, when $g$ is any simple function, with decreasing representation $g=\sum_ic_i1_{H_i}$, one has
$$\int_{X^*} g(x) p \cdot h(y) dm = \int_0^1\left(\int_Xg(x)d\mu_y\right)p \cdot h(y) dy=
\int_0^1\left(\sum_i c_i\mu_y(H_i) \right) p \cdot h(y)dy$$
thanks to Theorem \ref{fubini2}. 
But we have
\begin{eqnarray*}
\int_0^1\left(\sum_i c_i\mu_y(H_i) \right) p \cdot h(y)dy &=& 
\sum_i \int_0^1\left(c_i1_{H_i}(x)d\mu_y(x)\right) p \cdot h(y)dy= \\ &=& p \cdot \int_{X^*}g(x)h(y)dm,
\end{eqnarray*}
by virtue of (\ref{primostep}).
Finally, if $g$ is any bounded measurable function, it can be uniformly approximated by an increasing sequence of simple
 functions $(g_n)_n$; then, by the properties of the Choquet integral, one has that 
$$\lim_n \int_{X^*} g_n(x)p \cdot h(y)dm=\int_{X^*}g(x)p \cdot h(y)dm,$$
and finally
$$\int_{X^*}g(x)p \cdot h(y)dm=p \cdot \int_{X^*}g(x)h(y)dm$$
follows from the previous step.
\end{proof}

Our next goal is to prove that, in case $f:X^*\to \erre^n_+$ is sectional, and if $m$ is a decomposable fuzzy 
measure of a {\em special} type, then the set
$R(f):=\{\int_H f dm:H\in \mathcal{H}\}$
is convex. We need the following
\begin{definition}\label{convextype}\rm 
Let $m:\mathcal{H}\to \erre^+_0$ be a decomposable fuzzy measure. We shall say that $m$ is of {\em convex type}
 if the measures $\mu_y$ are uniformly filtering in the $\sigma$-algebra $\mathcal{A}$.
\end{definition}
We observe that $m$ is of convex type if all measures $\mu_y$ coincide with a semiconvex submeasure $\mu$ on 
$X$, or in the particular case of a finite number of sections as described in the Introduction.

We  shall also make use of the following Lemma.
\begin{lemma}\label{usotau}
In the situation described above, let $\tau:[0,1]\to [0,1]$ be any measurable mapping. Then, there exists a
 measurable set $A\in \mathcal{H}$ such that 
$$\mu_y(A_y)=\tau(y)\mu_y(X),$$
for all $y\in [0,1]$.
\end{lemma}
\begin{proof}
Since $0\leq \tau\leq 1$, it is possible to construct an increasing sequence $(s_k)_k$ of simple functions 
converging to $\tau$, and a decreasing sequence $(S_k)_k$ of simple functions, also converging to $\tau$. 
Moreover, these sequences can be based on the same partitions of $[0,1]$, built in a diadic way. So we can write
$$s_k=\sum_{i=0}^{2^k-1}c_i^k\ 1_{J_i^k},\ \ \ S_k=\sum_{i=0}^{2^k-1}C_i^k\ 1_{J_i^k},$$
where $c_i^k=\dfrac{i}{2^k},  C_i^k=\dfrac{i+1}{2^k}$, and 
$J_i^k=\tau^{-1} ([\dfrac{i}{2^k},\dfrac{i+1}{2^k}[) $ for all $k$ and $i=0...2^k-1$. 

Thanks to the convex-type hypothesis, there exists a filtering family $(X_t)_{t\in [0,1]}$ in $X$, satisfying (\ref{CM2.1}-i,ii,iii) simultaneously for all the measures $\mu_y$.
So we can set, for each $k$:
$$E_k:=\bigcup_{i=0}^{2^k-1}X_{c_i^k}\times J_i^k,\ \ \ F_k:=\bigcup_{i=0}^{2^k-1}X_{C_i^k}\times J_i^k.$$
 Clearly, the sets $E_k$ and $F_k$ belong to $\mathcal{H}$, and
we have $E_k\subset E_{k+1}\subset F_{k+1}\subset F_k$, for all $k$.
Moreover, we can see that
$$\mu_y(E_{k,y})=s_k(y)\mu_y(X),\ \ \mu_y(F_{k,y})=S_k(y)\mu_y(X)$$
for all $k$ and $y$.

Now, setting $E=\bigcup E_k,\ \ F=\bigcap F_k$, both $E$ and $F$ belong to $\mathcal{H}$, and $E\subset F$.

\noindent By monotonicity, we have then
$$\mu_y(E_y)\geq \sup_k\mu_y(E_{k,y})=\lim_ks_k(y)\mu_y(X)=\tau(y)\mu_y(X),$$
and
$$\mu_y(F_y)\leq \inf_k\mu_y(F_{k,y})=\lim_kS_k(y)\mu_y(X)=\tau(y)\mu_y(X).$$
Comparing these inequalities, and recalling that $E\subset F$, we can conclude that
$$\mu_y(E_y)=\mu_y(F_y)=\tau(y)\mu_y(X)$$
for all $y$. So, any of the sets $E$ or $F$ is as requested.
\end{proof}

\begin{theorem}\label{primoconvesso}
Let's assume that $m$ is a fuzzy measure of convex type, and that $f:X^*\to \erre^n_+$ is a sectional integrable function.
Then the set $R(f)$ is convex.
\end{theorem}
\begin{proof}
 As usual we shall identify $f$ with its (vector) section function, and fix any element $H\in \mathcal{H}$. We have, from Corollary \ref{fubini1}:
$$\int_H fdm= \int_0^1 f(y)\mu_y(H_y)dy.$$
More precisely, for any $y\in [0,1]$, set
$$\Lambda(y)=(f_1(y)\mu_y(H_y),...,f_n(y)\mu_y(H_y)):$$ clearly, $\Lambda$ is measurable, and 
$\int_H f dm=\int_0^1\Lambda(y)dy.$
Since $\mu_y(H_y)\leq \mu_y(X)$, the vector $\Lambda(y)$ lies in the line segment joining the origin $O$ with the vector 
$\mu_y(X)f(y)$. So, if we denote this segment by $S(y)$, it is clear that $\int_Hfdm\in \int_0^1 S(y)dy$,
where the latter is an Aumann integral, i.e. the set of all integrals of measurable selections from $y\mapsto S(y)$: 
from now on, we shall denote by $D$ this Aumann integral.
Since $D$ is clearly convex, if we prove that $R(f)=D$ 
the proof is complete.
\\
\noindent But we have already seen that $R(f)\subset D$, so it only remains to show the converse.
To this aim, let us fix any measurable selection $\sigma(y)$ from $S(y)$. Then, for all $y$ there exists a real number
 $\tau(y)\in [0,1]$ such that (componentwise)
$$\sigma_i(y)=\tau(y)\mu_y(X)f_i(y).$$
The mapping $\tau$ can be taken measurable: indeed, let $i$ be any index  for which $f_i(y)\neq 0$; 
then $\tau(y)=\dfrac{\sigma_i(y)}{\mu_y(X)f_i(y)}$
and this value is independent of $i$. Otherwise, if $f_i(y)= 0$
we can choose $\tau$ as in the previous case, since its value is immaterial.
\\
Now, since $0\leq \tau\leq 1$, thanks to Lemma \ref{usotau}, we can find a measurable set $E\in \mathcal{H}$ 
such that
$\mu_y(E_y)=\tau(y)\mu_y(X)$
for all $y\in [0,1]$, and so, componentwise:
$$\int_E f_idm=\int_0^1f_i(y)\mu_y(E_y)dy=
\int_0^1f_i(y)\tau(y)\mu_y(X)dx=\int_0^1\sigma_i(y)dy.$$
Since $\sigma$ was arbitrary, this shows that $D\subset R(f)$, and the proof is finished.
\end{proof}

\section{Applications to equilibria}
\label{walras}
We shall now introduce our economic model. 
For vector measurable functions $f = (f_1, \ldots f_n): X^* \rightarrow \mathbb{R}_+^n$  we consider the monotone integral, 
and we shall keep the notation $\int f dm$, as the vector $\int fdm= \left( \int f_1dm, \ldots ,\int f_ndm\right).$
Sometimes we shall denote by $a$ the generic element $(x,y)\in X^*$, and shall also use the notation 
$f\in L^1_{C}(m, \mathbb{R}_+^n)$ meaning that each component is in $L^1_{C}(m)$.\\

We define a \it pure exchange economy \rm to be a 4-tuple
$$ {\mathcal E}=\{ (X^*,\mathcal{H},m);~ \mathbb{R}_+^n;~ e;~~\{\succ_a\}_{a\in X^*} \},$$ where:\\
-  the \sl space of agents \rm is a triple $(X^*, \mathcal{H}, m),$ with $(X^*, \mathcal{H})$
a measurable space and $m$ is a fuzzy measure of convex type.
Moreover we shall require that each $m$ is a submodular and the ideal of $m$-null sets is stable under countable unions. Under these conditions (see \cite{primofuzzyar,ds}) the Choquet integral for scalar non-negative functions satisfies the following requirements:
\begin{description}
\item[($c_1$)] 
$\int_{X^*} (f+g)dm\leq \int_{X^*}f dm+\int_{X^*}g dm$
\item[($c_2$)] \label{c-2} 
If 
$ \int_A fdm \leq \int_A g dm$ for every $A\in \Hs$  then $ f\leq g \, m$-a.e.\\
\end{description}
{\bf We shall also assume here that $\mu_y(X)=m(X^*)=1$ for all $y\in [0,1]$.}\\
-the target space $\erre ^n$ is the \sl commodity space\rm, and its positive cone  
$\mathbb{R}_+^n$ is called the \sl consumption set\rm\ of each agent;\\
- $e:X^* \rightarrow \mathbb{R}_+^n$, $e\in L^1_{C} (m, \mathbb{R}_+^n)$ is the \sl\ initial
endowment  density \rm and $e(a):=e(x,y) \gg 0$ \,$m$-a.e.; we  shall always assume that $e$ is sectional.\\
- $\{\succ_a\}_a$ is the {\sl preference relation} associated to the generic agent
$a\in X^*$.

\noindent
Let us introduce the classical concepts of equilibrium theory in this new setting.
\\
- An {\em allocation} is a measurable function $f:X^* \longrightarrow \mathbb{R}_+^n$; an allocation is \sl feasible\rm\ if
$$\displaystyle\int_{X^*} f \,dm=\int_{X^*} e \,dm.$$
- A {\em price}
is any element $p\in \mathbb{R}_+^n \setminus \{0\}$.\\
- The {\em  budget set} of an agent $a\in X^*$ for the price $p$ is
$B_{p}(a)=\{x\in \mathbb{R}_+^n:\ p x\le p e(a)\}.$\\
- A {\em coalition} is any measurable subset $S$ of $X^* $ 
with $m(S)>0$.\\
- We say that the coalition $S$ can {\em  improve} the
allocation $f$ if there exists an allocation $g$ such that
\begin{description}
\item[($i_1$)] $g(a)\succ_a f(a)\ m$-a.e. in $S$;
\item[($i_2$)] $\displaystyle\int_Sg \,dm =\int_S e \,d m$.
\end{description}
- \rm The {\em core} of ${\mathcal E}$, denoted by  $C({\mathcal E})$, is the set
of all the feasible allocations that cannot be improved by any coalition.\\
- We say that the coalition $S$ \em strongly improves \rm\ the
allocation $f$ if there exists an 
allocation $g$ such that
\begin{description}
\item[($i_1$)] $g(a)\succ_a f(a)\ m$-a.e. in $S$;
\item[($i^{\prime}_2$)] $\displaystyle\int_{S_y}g(\cdot,y) \,d\mu_y =\int_{S_y} e(\cdot,y) \,d \mu_y$ \ for a.e. $y\in [0,1]$.
\end{description}
-  The {\em  large core} of ${\mathcal E}$, denoted by $\LCEs$, is the set
of all the feasible allocations that cannot be strongly improved by any coalition. Of course, $\CEs\subset \LCEs$.\\
-\rm\ A \it Walras equilibrium\rm\
of ${\mathcal E}$ is a pair
$(f,p)\in L^1_C(m, \mathbb{R}_+^n)\times(\mathbb{R}_+^n \setminus \{0\})$
such that:
\begin{description}
\item[($w_1$)] $f$ is a feasible allocation;
\item[($w_2$)] $f(a)$ is a maximal element of $\succ_a$ in the
budget set $B_{p}(a)$, 
(namely $f(a) \in B_p (a)$ and $x \succ_a f(a)$ implies $p \cdot x > p \cdot e(a)$)
for $m$-almost all $a\in X^*$.
\end{description}
- A {\em walrasian allocation}  is a feasible allocation $f$ such
that there exists a price $p$ so that the pair $(f,p)$ is a Walras
equilibrium. $W({\mathcal E})$ is the set of all the walrasian allocations of ${\mathcal E}$.\\

Our aim is to obtain relations between  Walras equilibria  ${\mathcal W}(\mathcal{E})$ 
and core of an economy ${\mathcal C}({\mathcal E})$.

\begin{ass}\label{leA} 
\begin{description}
\item[(\ref{leA}.1)] ({\sl Perfect competition})\  $m$ is a fuzzy  measure of convex type, and the corresponding  
measures $\mu_y$ are sub-additive.
This condition describes an economy where the big coalition is the average of 
many autonomous sections.
\item[(\ref{leA}.2)] $e:X^* \rightarrow (\mathbb{R}_+^n)^{\circ}$ is sectional, and $\varphi^e$ will denote
 its (vector) section function.
Observe that this implies that the aggregate initial endowment
\mbox{$\displaystyle\int_{X^*} e(a)\,dm\in (\mathbb{R}_+^n)^{\circ}$}. 
Let $\lambda: \mathcal{H} \rightarrow \mathbb{R}_+^n$
defined by $\lambda(H) = \int_H e dm$. Then $\lambda$ is a fuzzy subadditive measure.
\item[(\ref{leA}.3)] In $\mathbb{R}_+^n$ there exist  preorders $\succ_a$, $a\in X^*\times [0,1]$, 
that satisfy the following:
\begin{description}
 \item[(\ref{leA}.3a)] ({\sl Monotonicity}) for every $x\in \mathbb{R}_+^n$ and every 
$v\in  \mathbb{R}_+^n \setminus\{0\},$ $x+v\succ_a x$ for all
$a\in X^*$.
 \item[(\ref{leA}.3b)] {\em continuity}, namely for all $x \in (\mathbb{R}_+^n)^{\circ}$ the set 
$\{y \in \mathbb{R}_+^n: y \succeq_a x
 \}$ is closed in $\mathbb{R}_+^n$ for all $a \in X^*$;
\end{description}
  and  such that for every $y$ and  $a \in X\times {y}$, $\succeq_a $ is the same preorder depending only on $y$,
 i.e. \ $x \succeq_a x'$ can be written as
$x \succeq_y x'$. In other words, in each coalition $E_y:=X\times\{y\}$, agents
share the same initial endowment and the same preference
criterion.
\end{description}
\end{ass}
In order to study relations between ${\mathcal C}({\mathcal E})$ and  ${\mathcal W}(\mathcal{E})$,
we observe that
\begin{theorem}\label{WESCES}
Under the assumptions  \ref{leA} the inclusion $\CEs \supset \WEs$ holds true.
\end{theorem}
\begin{proof} 
The proof is analogous to \cite[Theorem 3.2]{primofuzzyar} and it is reported here for the sake of completeness.
Let $f \in \WEs \setminus \CEs$.
Then there exist a coalition  $S$ and a feasible allocation $g$ such that $m$-a.e. in $S$
$g(a) \succ_a f(a)$~ and ~ $\int_S g dm= \int_S e\ dm$.
On the other side there exists a price  $p$ for which $f(a)$ is $\succ _a$ maximal in  $B_p (a)$ $m$-a.e.
in $S$. \par \noindent
Consequently, setting $S_1 = \{ a \in S : p \cdot g(a) \leq p \cdot e(a) \},$ it should be \mbox{$m(S_1) = 0$,}
otherwise the function
$\widetilde{f} =g 1_{S_1} + f 1_{S \setminus S_1}$ would contradict the maximality of $f$.\\
Hence $m$-a.e. in $S$ one has:
\begin{eqnarray*}
p \cdot g(a) &=& \sum_{i=1}^n p_i g_i(a) > \sum_{i=1}^n p_i e_i (a) = p \cdot e(a)
\end{eqnarray*}
whence,  by Proposition \ref{perscalari}
\begin{eqnarray}\label{uno}
&& 
\int_S p \cdot g(a) dm= \int_S \sum_{i=1}^n p_i g_i(a) dm>
\int_S \sum_{i=1}^n p_i e_i (a) dm=
p \cdot \int_S e dm.
\end{eqnarray}
Thus
\begin{eqnarray*}
\int p \cdot g d m> p \cdot \int e dm.
\end{eqnarray*}
On the other side, as we have assumed that $g$ improves $f$, from subadditivity we have:
\begin{eqnarray*}
\int_S p \cdot g dm\leq p \cdot \int_S g(a) dm= p \cdot \int_S e dm
\end{eqnarray*}
thus contradicting (\ref{uno}).
\end{proof}

\begin{ass}\label{ass-gamma} 
 Suppose now that $f$ is sectional, so that
$f(a)=f(x,y) = f(y)$ and consider the multifunction 
\begin{eqnarray}\label{eq:gamma}
&& \Gamma_f (a) :=\Gamma_f(x,y)= \{t \in \mathbb{R}_+^n: t \succeq_a f(a)\}= 
\{t \in \mathbb{R}_+^n: t \succeq_y f(y)\} = C_y,
\end{eqnarray}
 where the ${C_y}'s$ are convex,
closed and contain the sets $u + (\mathbb{R}_+^n)^{\circ}$ when $u \in C_y$. 
The class of its Choquet integrable selections is
\(S^*_{\Gamma_f}= \{ \psi \in L^1_{C} (m, \mathbb{R}_+^n) \mbox{~with~} \psi (a) 
\in \Gamma_f (a) \mbox{~for~} m - \mbox{~a.e.~} a\in X\}.\)\\

So $\Gamma_f$ contains as selections all functions 
that are $\mu_y$-a.e. constant in $X\times\{y\}$ (the constant must be an element of $C_y$).
So all integrable functions of the type $\gamma(x,y)= c(y), c(y) \in C_y$, 
are Choquet integrable selections of $\Gamma$.\\
\end{ass}

Let \begin{eqnarray}\label{I-f}
I_f := \left\{ z = \int_H s dm- \lambda(H), ~~\forall~ H \in \mathcal{H},
~~s \in S^*_{\Gamma_f} \right\}.
\end{eqnarray}
Now, in order to prove the convexity of $I_f$ we need some preliminary results;
the first is a density result of the multivalued integral of $\Gamma$.

\begin{lemma}\label{gsemplice}
If $s \in S^*_{\Gamma_f}$ then, for every $A \in \mathcal{H}$, there exists a sectional
selection $g \in S^*_{\Gamma_f}$ such that
\[\int_A s dm= \int_A g dm.\]
\end{lemma}
\begin{proof}
Let $A \in \mathcal{H}$ and $s   \in S^*_{\Gamma}$ be fixed.
For each $y$, let us define
$$\varphi(y)=\dfrac{1}{\mu_y(A_y)}\int_{A_y}s(x,y)d\mu_y(x),$$
in case $\mu_y(A_y)>0$. Otherwise we can set $\varphi(y)$ equal to any arbitrary selection of $y\mapsto C_y$.
Indeed, we see that
\begin{eqnarray*}
\int_A s dm  &=& \int_K\left(\int_{A_y}s(x,y)d\mu_y(x)\right) dy+
\int_{K^c}\left(\int_{A_y}s(x,y)d\mu_y(x)\right) dy=
\\ &=& 
\int_K\left(\int_{A_y}s(x,y)d\mu_y(x)\right) dy,
\end{eqnarray*}
where $K$ denotes the set of all $y\in [0,1]$ such that $\mu_y(A_y)>0$.
Then, setting $g(x,y)=\varphi(y)$, we have
$$\int_A g(x,y) dm=\int_0^1\varphi(y)\mu_y(A_y) dy$$
thanks to Corollary \ref{fubini1}, and so
$$\int_ A g(x,y) dm= \int_K \varphi(y)\mu_y(A_y) dy=
\int_K\left(\int_{A_y}s(x,y)d\mu_y(x)\right) dy=\int_A s dm.$$
It only remains to show that $\varphi(y)\in C_y$, for each $y\in K$. 
To this aim, fix $y\in K$ and assume by contradiction that the quantity
$\varphi(y)$ does not belong to $C_y$. Then, by the Separation Theorem,
there exist a positive element $p\in \erre^n$ and a real number $a$ such that 
$$p\cdot \varphi(y)<a,\ \ \ \inf\{p\cdot z: z\in C_y\}\geq a.$$
But then, by subadditivity, we get
$$a>p\cdot \varphi(y)=\sum_{i=1}^np_i\dfrac{1}{\mu_y(A_y)}\int_{A_y}s_i(x,y)d\mu_y(x)
\geq\dfrac{1}{\mu_y(A_y)} \int_{A_y}\sum_{i=1}^np_is_i(x,y)d\mu_y(x).$$
Now, for every $x\in X$ we have $s(x,y)\in C_y$, and so 
$$\sum_{i=1}^np_is_i(x,y)\geq a:$$
from this we deduce
$$ a>p\cdot \varphi(y)\geq\dfrac{1}{\mu_y(A_y)} \int_{A_y}\sum_{i=1}^np_is_i(x,y)d\mu_y(x)\geq \dfrac{1}{\mu_y(A_y)}\int_{A_y}a\ d\mu_y(x)=a.$$
This is clearly absurd, and the assertion is proved.
\end{proof}

Using Lemma \ref{gsemplice} we now prove the main convexity theorem.
\begin{theorem}\label{1.3.1.cond-c}
The set ${I_f}$ is convex, when $f$ satisfies Assumption \ref{ass-gamma}.
\end{theorem}
\begin{proof}
 Thanks to Lemma \ref{gsemplice}, in the definition of $I_f$ the mapping $s$ can be taken sectional, without loss of generality. So we always assume this, and write the values of $s$ just as $s(y)$, subject to the condition $s(y)\in C_y$ for all $y$.
Our aim is to prove that
$$I_f =\widehat{\int}_0^1 I(y)dy,$$
where $I(y)$ is the convex cone $\bigl(C_y-e(y)\bigr)\cdot[0,1]$, and the multivalued integral is meant as follows:
$$\widehat{\int}_0^1 I(y)dy=\left\{\int_0^1 (s(y)-e(y))\tau(y)dy: s\in S^*_{\Gamma_f}, \tau(y)\in [0,1],\ \ s,\tau
 \ {\rm integrable }\right\}.$$
Indeed, if $z\in I$ there exist an integrable selection $s$ and a set $A\in \mathcal{H}$ such that
$$z=\int_A s dm=\int_0^1(s(y)-e(y))\mu_y(A_y)dy.$$
By definition of $m$, the mapping $y\mapsto \mu_y(A_y)$ is measurable, and obviously bounded between $0$ and $\mu_y(X)=1$, so clearly $z\in \displaystyle{\widehat{\int}_0^1I(y)dy}$.
\\
Conversely, let us take an element $w\in \displaystyle{\widehat{\int}_0^1I(y)dy}$: then there exist a measurable selection $s$ and a measurable mapping $\tau:[0,1]\to [0,1]$, such that
$$w=\int_0^1(s(y)-e(y))\tau(y)dy.$$
Thanks to Lemma \ref{usotau}, we see that there exists a measurable set $A\in \mathcal{H}$ such that $\tau(y)=\mu_y(A_y)$ for all $y$, and so
$$w=\int_0^1(s(y)-e(y))\mu_y(A_y)dy=\int_A(s-e) dm\in I_f.$$
This proves also the reverse inclusion. Now, we shall prove that  
$\displaystyle{\widehat{\int}_0^1I(y)dy}$ is convex.
So, take two elements $w_1$ and $w_2$ from $\displaystyle{\widehat{\int}_0^1I(y)dy}$, and fix any positive number $c\in ]0,1[$.
By definition, there exist  measurable selections $s_i$ and measurable mappings $\tau_i:[0,1]\to [0,1]$, $i=1,2$,
such that
$$w_i=\int_0^1(s_i-e_i)(y)\tau_i(y) dy,$$
$i=1,2$. Define
$$K:=\{y\in [0,1]:\tau_1(y)+\tau_2(y)\neq 0\}.$$
Then we have
\begin{eqnarray*}
cw_1+(1-c)w_2  &=& \int_K\left[c(s_1-e)(y)\tau_1(y)+(1-c)(s_2-e)(y)\tau_2(y)\right]dy  =
\\
&=& \int_K(c\tau_1(y)+(1-c)\tau_2(y)) \cdot \left[\dfrac{c\tau_1(y)(s_1(y)-e(y))}{c\tau_1(y)+(1-c)\tau_2(y)}+ \right.
\\
 &+&  \left. 
\dfrac{(1-c)\tau_2(y)(s_2(y)-e(y))}{c\tau_1(y)+(1-c)\tau_2(y)}\right]dy.
\end{eqnarray*}
Let us set
\[\tau(y)=\left\{\begin{array}{ll}
c\tau_1(y)+(1-c)\tau_2(y),&  y\in K\\
0,&  y\notin K\end{array}\right.
\]
and
\[
s(y)=\left\{\begin{array}{ll}
\dfrac{c\tau_1(y)s_1(y)}{c\tau_1(y)+(1-c)\tau_2(y)}+
\dfrac{(1-c)\tau_2(y)s_2(y)}{c\tau_1(y)+(1-c)\tau_2(y)} & \ y\in K\\
& \\
s_1(y) &  y\not\in K.  \end{array} \right.
\]
Then, it is clear that $\tau$ is a measurable mapping with values in $[0,1]$ and $s$ is a measurable function, 
such that $s(y)\in C_y$
for all $y$: moreover, it is clear from the previous calculations that
$$cw_1+(1-c)w_2=\int_0^1(s(y)-e(y))\tau(y)dy\in \widehat{\int}_0^1  I(y)dy.$$
So, we have proved also that the set $\displaystyle{\widehat{\int}_0^1 I(y)dy}$ is convex, which concludes the proof.
\end{proof}

In analogy with our previous notation, let $R(\lambda)$ denote the range of the set function $\lambda$,
namely $R(\lambda ) = \lambda (\mathcal{A})$.

\begin{lemma}\label{separazione}
If $f\in \LCEs$, there exists $p \in \mathbb{R}_+^n$, $p \neq 0$ such that $p \cdot x \geq 0$ for 
all $x \in \overline{I}_f$.
\end{lemma}
\begin{proof}
We shall first prove that
$I_f \cap (-\mathbb{R}_+^n) = \{ 0\}.$
Indeed, assume by contradiction that there exists
$z\in I_f \cap (-\mathbb{R}_+^n)$ with $z\not= 0$. Then there exist a coalition $A\in \mathcal{H}$ and
a Choquet integrable selection $s\in S^1_{\Gamma_f}$ such that
$$z = \int _A sd m- \int _A edm\in (-\mathbb{R}_+^n).$$
Thanks to Lemma \ref{gsemplice} we can and do assume that $s$ is sectional.
Then 
immediately we get $m(A) > 0$ (otherwise both $\displaystyle{\int _A sdm= 0}$ and
$\displaystyle{\int _A edm= 0}$ whence $z = 0$).
We observe that $z=\displaystyle{\int_0^1z_y dy}$, where $z_y=\displaystyle{\int_{A_y}s(x,y)d\mu_y} - e(y)\mu_y(A_y)=
(s(y)-e(y))\mu_y(A_y)$. 
\\
Now, let $J:=\{y: z_y\neq 0\}$. Of course, $\lambda(J)\neq 0$ otherwise $z=0$, and $z=\displaystyle{\int_{J}z_ydy}$. 
Now, for each $y\in J$, we have $\mu_y(A_y)>0$ (otherwise $z_y=0$), and let us define
$A':=\bigcup_{y\in J}A_y$:
the set $A'$ is a measurable subset of $X^*$, since $J$ and $A$ are. Finally let us 
set
$$s_0(x,y):=s(y)-\dfrac{z_y}{\mu_y(A_y)},$$
for all $y\in J$ and $x\in X$.
Since $z_y\in (-\erre^n_+)\setminus \{0\}$ for each $y$, we see that the allocation $s_0$ satisfies 
$s_0(a)\succ_a s(a)\succeq_af(a)$ $\mu$-a.e. in $A'$, and moreover 
$$\int_{A'_y}s_0 d\mu_y=\int_{A'_y}s(y)d\mu_y-z_y=e^y\mu_y(A_y)$$
holds true, for all $y\in J$. 
Moreover, if $y\notin J$, we have by definition $A'_y =\emptyset$, and so $\displaystyle{\int_{A'_y}s_0 d\mu=0=\int_{A'_y} e d\mu}.$
So we have proved that the coalition $A'$ strongly improves $f$ by the allocation $s_0$. But this is impossible, since $f\in \LCEs$.\\
In conclusion $I_f \cap (-\mathbb{R}_+^n) = \{ 0\}$ and hence $\overline{I}_f \cap (-\mathbb{R}_+^n)^o = \emptyset$.
Since both sets are convex, and the second one has non-empty interior, we can  apply the Strong Separation Theorem, and
determine some $p \in \erre ^n$ $p\not= 0$ such that $p \cdot z \geq 0$ for all $z \in \overline{I}_f$.
\\
It only remains to prove that
$p\in \mathbb{R}_+^n.$
Indeed, we have that $(\mathbb{R}_+^n)^o \subset I_f;$ in fact if $z\in (\mathbb{R}_+^n)^o$, then the allocation
 $\displaystyle{\psi =
\dfrac{z}{m(X^*)} + f}$ is in $S^*_{\Gamma_f}$ and
\begin{eqnarray*}
\int _{X^*} \psi dm- \int _{X^*} e dm= \int _{X^*} fdm+ z - \int _{X^*} edm= z
\end{eqnarray*}
since $f$ is feasible. Then  
$p\cdot z \geq 0$ for every $z\in (\mathbb{R}_+^n)^o,$ whence necessarily
$p\in \mathbb{R}_+^n$. 
\end{proof}

We shall now prove that:
\begin{lemma}\label{strassen} 
{\rm \bf (Strassen)}
If $f$ is sectional and $f \in \LCEs$, then $p\cdot e(a)\leq p\cdot f(a)$  $m$-a.e., where $p$ is as in Lemma {\rm \ref{separazione}}
\end{lemma}
\begin{proof}
  Let $B:=\{y\in [0,1]: p\cdot f(y)<p\cdot e(y)\}$. Since
$\mu_y (X) = m(X^*) = 1$ and 
 $f$ and $e$ are sectional, we have
$$0<\int_B(p\cdot e-p\cdot f)(y)dy=\int_{X\times B}(p\cdot e-p\cdot f) dm.$$
So, by Corollary \ref{fubini1} and Remark \ref{add}, it follows
\begin{eqnarray*}
0 &<&\int_{X\times B} p\cdot (e-f) dm=
p\cdot \int_{X\times B}e dm -p\cdot \int_{X\times B} f dm=-p\cdot z,\end{eqnarray*}
where $z=\int_{X\times B} f dm  -\int_{X\times B} e dm$ is an element of $I_f$. Therefore $p\cdot z<0$, which is in contrast with Lemma \ref{separazione}.  
\end{proof}

\begin{theorem}\label{finale}
Under the previous assumptions, 
if $f$ is a  sectional allocation,
then $f \in \LCEs$ if and only if $f \in \WEs$.
\end{theorem}
\begin{proof}
We have already seen that $\WEs\subset \CEs$, thanks to Theorem \ref{WESCES}.
To prove the converse inclusion fix
 $f \in \LCEs$. 
By Lemma \ref{strassen}  $m$
a.e. in $X^*, p \cdot e(a) \leq p \cdot f(a)$.
 We shall now prove that the previous inequality is in fact
an equality. 
For $A \subset X^*$, being $f$ feasible and since $p\cdot f -p\cdot e \geq 0$ $m$-a.e., we derive  by Proposition \ref{perscalari} and (\ref{sec-add})
\begin{eqnarray*}
 0\leq  \int_A p\cdot( f - e) d m \leq \int_{X^*}  p\cdot( f - e) dm 
= p\cdot \left(\int_{X^*}fd\mu-\int_{X^*}e dm \right).
\end{eqnarray*}
Applying ($c_2$)  (pag. \pageref{c-2}), 
we  get $p\cdot f = p\cdot e$
 $m$-a.e. in $X$.\\
The remaining part of the proof is exactly the same as that of
\cite[Theorem 2.1.1, pag 133 ff]{Hilde}  
since preferences are
assumed to be monotone and continuous.  
\end{proof}

Assume now that the preferences have the following structure:
\begin{itemize}
\item for every $y\in [0,1]$, there exists a subset $J_y$ of  $\{1, 2, \ldots, n\}$ such that:
\begin{itemize}
\item[i)] for every $u, v \in \erre^n_+$, $u \succ_{a=(x,y)} v \Longleftrightarrow u_j > v_j, j \in J_y$;
\item[ii)] For every $j\in \{1,...,n\}$, the set $A_j:=\{y\in [0,1]:j\in J_y\}$ is measurable.
\end{itemize}
This means that within each coalition $E_k$ only the items of the $k$-th list $J_k$ are considered, 
in order to decide whether a bundle is preferred to another. Observe that such assumption does not 
fulfill monotonicity, in the sense of (\ref{leA}.3b), but it satisfies the more demanding form
\item for every $x \in \mathbb{R}_+^n$,  $z \in (\mathbb{R}_+^n)^0$, then $x + z \succ_a x$ for every $a \in X$.
\end{itemize}
However Lemma \ref{separazione} remains true: one has only to note that $I \cap (-\mathbb{R}_+^n)^0 = \emptyset$ 
with the same proof.

\begin{proposizione}\label{cor-finale}
Under Assumption \ref{leA}, $e  \in \WEs$.
\end{proposizione}
\begin{proof}
 The proof is analogous to that of \cite[Proposition 3.13]{primofuzzyar} when we apply Theorem \ref{finale}.
It is enough to prove that $e  \in \CEs$ and then apply Theorem \ref{finale}. Assume by contradiction that
$e  \not\in \CEs$; then there exists a pair $(f,S)$ that improves $e$, namely
\begin{itemize}
\item[\ref{cor-finale}.a)] $f \succ_a e$, when $a \in S$;
\item[\ref{cor-finale}.b)] $\iS f dm= \iS e dm$.
\end{itemize}
We note that, from (i) and (ii) above, it follows that there exists $k\in \{1,...,n\}$ such that $\lambda(A_k)>0$.
From \ref{cor-finale}.a), we have for the $k$-th entries of $f$ and $e$,
$f_k(a) > e_k(a), a \in S$. Hence by ($c_2$), there holds
$$ \iS f_k dm> \iS e_k dm$$
that contradicts \ref{cor-finale}.b). 
\end{proof}
\section*{Conclusions}
The Choquet integral over a product space $X^*$ and with values in $\mathbb{R}_+^n$ has been studied with respect to a fuzzy measure.
Under suitable conditions a Fubini theorem is obtained  and these results are used to find equilibria in a pure exchange economy
$ {\mathcal E}=\{ (X^*,\mathcal{H},m);~ \mathbb{R}_+^n;~ e;~~\{\succ_a\}_{a\in X^*} \},$ where 
  the space of agents  is a triple $(X^*, \mathcal{H}, m),$ with $(X^*, \mathcal{H})$
a measurable space and $m$ is a fuzzy measure of convex type.
\\
If the target space is infinite dimensional  vector lattices are candidates for the space of goods.
 In this framework one could  consider also the methods of integration given in \cite{bms13,SC,bcs2015,dpm,cs2014,cs2015,ds}.
\section*{Acknowledgement}
This research was partially supported by GNAMPA -- INDAM (Italy)
 Grant  N.  U UFMBAZ2017/0000326, by University of Perugia --
Dept. of Mathematics and Computer Sciences -- Grant Nr 2010.011.0403 and Grant APVV-14-0013.\\


\begin{thebibliography}{99}
\bibitem{AM1}
 L. Angeloni - A. Martellotti {\em A separation result with applications to
Edgeworth equivalence in some infinite dimensional setting},  Commentationes Math. \bf 44 \rm (2004) 227-243.

\bibitem{AM4} 
L. Angeloni - A. Martellotti {\em Core-Walras equivalence in finitely additive economies with
extremely desirable commodities},  Mediterr. J. Math.   {\bf 4},   (2007). 87-107.

\bibitem{AR} 
T.E. Armstrong - M.K. Richter, {\em The Core-Walras equivalence,
\em J. Econom. Theory, \rm  {\bf 33}, (1984),  116-151.}

\bibitem{bb2017}
G. Barbieri - A. Boccuto, {\em On extensions of $k$-subadditive lattice group valued capacities}, Italian Journal of Pure and Applied Math. {\bf 37},  (2017), 387-408.

\bibitem{B}
 A. Basile, {\em Finitely additive nonatomic coalition production economies: core-Walras equivalence,
\rm Internat. Econom. Rev. {\bf 34}, (1993),  983-995.}

\bibitem{BG} 
A. Basile - M.G. Graziano, {\em On the Edgeworth's conjecture in finitely additive
economies with restricted coalitions, \rm  J, Mathematical  Economics,  \rm  {\bf 36} (3), (2001),  219-240.}

\bibitem{BDG1}
A. Basile - C. Donnini - M.G. Graziano, {\em Core and equilibria in coalitional asymmetric information economies, \rm J. of Mathematical Economics, {\bf 45},  (2009), 293-307.}

\bibitem{BDG13} 
A. Basile - C. Donnini - M.G. Graziano, {\em  Core equivalences for equilibria supported by non-linear prices}, Positivity {\bf 17} (3), (2013),  621–653.

\bibitem{BM2000} 
P. Benvenuti - R. Mesiar, {\em A note on Sugeno and Choquet integrals} in: Proceedings of the IPMU, Madrid, 2000, pp. 582–585.

\bibitem{BM2000a} 
 P. Benvenuti -  R. Mesiar, {\em Integrals with respect to a general fuzzy measure} in: M. Grabisch et al. Fuzzy Measures and Integrals,
Theory and applications Stud. Fuzziness Soft Computer, 40, (2000) Physica-Verlag, Heidelberg, pp. 205–232.

\bibitem{BMV2002} 
P. Benvenuti - R. Mesiar - D. Vivona, {\em Monotone set functions-based integrals}, in: E. Pap (Ed.), Handbook of Measure Theory, vol. II,
Elsevier, (2002),  1329-1379.

\bibitem{bcs2015}
A. Boccuto -  D. Candeloro - A.R. Sambucini,  {\em
Henstock multivalued integrability in Banach lattices with respect to pointwise non atomic measures},
Atti Accad. Naz.
Lincei Rend. Lincei Mat. Appl.
{\bf 26}, (2015), 363-383

\bibitem{bms13}  
A. Boccuto - A.M. Minotti - A.R. Sambucini,  {\em Set-valued Kurzweil-Henstock integral in Riesz space setting}, PanAmerican Mathematical Journal, {\bf 23} (1) (2013), 57--74.

\bibitem{bs9596}  
A. Boccuto - A.R. Sambucini,  \textit{ On the De Giorgi-Letta integral with respect to means with values in Riesz spaces}, Real Analisys Exchange {\bf 21} (2), (1995/96), 793--810. 

\bibitem{bs97} 
A. Boccuto - A.R. Sambucini,  \textit{ Comparison between different types of abstract integral in Riesz spaces}, Rend. Circ. Matematico di Palermo, Serie II, {\bf 46} (1997), 255--278.
 
\bibitem{CM1978}
D. Candeloro -  A. Martellotti, {\em Su alcuni problemi  relativi a misure scalari subadditive e applicazioni al caso dell'additività finita},  Atti Sem. Mat. Fis. Univ.
Modena, ~ {\bf 27},  (1978), 284-296.

\bibitem{CM1979}
 D. Candeloro - A. Martellotti, {\em  Sul rango di una massa
vettoriale}, Atti Sem. Mat. Fis. Univ. Modena,  {\bf 28}, (1979), 102-111.

\bibitem{cs2014}
D. Candeloro - A. R. Sambucini, {\em Order-type Henstock and Mc Shane integrals in Banach lattices setting}, arXiv:1401.7818 [math.FA], Sisy 20014- IEEE 12th International Symposium on Intelligent Systems and Informatics, Subotica - Serbia; 09/2014 

\bibitem{cs2015}
D. Candeloro - A.R. Sambucini, {\em Comparison between some norm and order gauge integrals in Banach lattices}, PanAmerican Mathematical Journal, {\bf 25} (3), (2015), 1-16.

\bibitem{Choquet}
G. Choquet, {\em Theory of capacities}; Ann. Inst. Fourier, {\bf 5} (1953), 131-295.


\bibitem{colreg} 
G. Coletti - G. Regoli, {\em Sulla funzione di distribuzione di una misura di probabilità finitamente additiva}, Rend. Mat. Appl.  {\bf 1} (7), (1981), 319-329.

\bibitem{cv2013}
G. Coletti -  B. Vantaggi,  {\em Conditional not–additive measures and fuzzy sets},  in: Proc.
ISIPTA 2013, pp. 67-76.

\bibitem{cpv2014}
G. Coletti - D. Petturiti - B.  Vantaggi, 
{\em Choquet expected utility representation of preferences on generalized lotteries},  Information processing and management of uncertainty in knowledge-based systems. Part II, 444–453,
Commun. Comput. Inf. Sci., 443, Springer, (2014).

\bibitem{ds}
O. Delgado - E.A. Sánchez Pérez,
{\em Choquet type $L^1$-spaces of a vector capacity},
Fuzzy Sets  and Systems {\bf 327}, (2017), 98-122.

\bibitem{D} 
D. Denneberg, {\em Non additive measures and integrals, \rm Kluwer Acad. Publ. Ser. B Vol. 27, (1994).}

\bibitem{dpm}
L. Di Piazza - V. Marraffa, {\em  
Pettis integrability of fuzzy mappings with values in arbitrary Banach spaces}, Math. Slov {\bf 66} (5), , (2016) 1119-1138, arXiv:1710.04124v1.

\bibitem{es}
E. Hewitt - K. Stromberg, {\em Real and abstract analysis}, Springer-Verlag  Berlin (1969).

\bibitem{Hilde}
W. Hildebrand, {\em Core and Equilibria of a Large
Economy, \rm Princeton University Press, Princeton, (1974).}

\bibitem{JKK}
 L.C. Jang - B. M. Kil  - J. S. Kwon, {\em Some properties of Choquet integrals of set-valued functions}, Fuzzy Sets and Systems,  {\bf 91}, (1997), 95-98.

\bibitem{JK} 
L.C. Jang - J. S. Kwon, {\em On the representation of Choquet integrals of set valued functions and null sets, \rm Fuzzy Sets and  Systems, {\bf 112}, (2000), 233-239.}

\bibitem{L}
 S. J. Leese, {\em Multifunctions of Souslin type, \rm Bull. Austr. Math. Soc., {\bf 11}, (1974), 395-411.}

\bibitem{LIMP} 
J. Li - R. Mesiar - E. Pap, {\em The Choquet integral as Lebesgue integral and related inequalities},  Kybernetika, {\bf 46} (6), (2010),  198-1107.

\bibitem{LIMP2}
 J. Li - R. Mesiar - E. Pap, {\em Atoms of weakly null-additive monotone
 measures and integrals}, Information Sciences \textbf{257}, (2014), 183-192.

\bibitem{mesiar}
 R. Mesiar, {\em Choquet-like integrals}, J. Math. Anal. Appl. {\bf 194}, (1995), 477-488.

\bibitem{mes-sip}
 R. Mesiar - J. \v{S}ipo\v{s}, {\em A theory of fuzzy measures: integration and its additivity}, International Journal of General Systems, {\bf 23} (1),(2007), 49-57, Doi: 10.1080/03081079408908029

\bibitem{m-altri}
R. Mesiar - A. Kolesárová - H.  Bustince -G.P. Dimuro - B.C. Bedregal, {\em  Fusion functions based discrete Choquet-like integrals},  European J. Oper. Res. {\bf 252} (2), (2016), 601-609. 

\bibitem{MuSu}
 T. Murofushi - M. Sugeno, {\em A Theory of Fuzzy measures: Representations, the Choquet Integral and Null Sets, 
\rm J. Math. Anal. and Appl., {\bf 159}, (1991), 532-549.}

\bibitem{na}
L. Nachbin, {\em Topology and order}, Van Nostrand (1965).

\bibitem{nt} 
Y. Narukawa - V. Torra,  {\em Fuzzy measures and Choquet integral on discrete spaces}, Computational Intelligence, Theory and Applications, Volume 33 of the series Advances in Soft Computing pp. 573-581 (2005).

\bibitem{pap1}
 E. Pap, {\em Pseudo-additive measures and their applications},  Handbook of Measure Theory, E. Pap ed., Elsevier (2002).

\bibitem{primofuzzyar}
 A.R. Sambucini, {\em The Choquet integral with respect to fuzzy measures and applications}, to appear in  
 Math. Slov. {\bf 67} 6, (2017) Doi: 10.15151/ms-2017-0049

\bibitem{SC}
 C. Stamate - A. Croitoru {\em Non linear integrals, properties and relationships}, Recent Advances in Telecommunications, Signal and Systems, ISBN: 978-1-61804-169-2,  WSEAS Press,  118-123

\bibitem{V} 
K. Vind, {\em Edgeworth Allocations in Exchange Economy with
many Traders, \rm Int. Econom. Review, {\bf 5},(1964), 165-177.}

\bibitem{z}
J. Zhang, {\em Subjective ambiguity, expected utility
and Choquet expected utility}, Economic Theory
{\bf 20}, (2002), 159-181.
\end{thebibliography}
\end{document}